\documentclass[12pt,leqno]{article}
\usepackage{amsfonts,amsthm,amsmath}
\newtheorem{thm}{Theorem}[section]
\newtheorem{prop}[thm]{Proposition}
\newtheorem{lem}[thm]{Lemma}
\newtheorem{cor}[thm]{Corollary}
\theoremstyle{remark}

\newtheorem{rem}[thm]{Remark}

\newcommand{\ZZ}{\mathbb{Z}}

\newcommand{\allone}{\mathbf{1}}
\DeclareMathOperator{\supp}{supp}

\DeclareMathOperator{\wt}{wt}

\begin{document}

\title{On the Residue Codes of
Extremal Type~II $\ZZ_4$-Codes of Lengths $32$ and $40$}

\author{Masaaki Harada\thanks{Department of Mathematical Sciences,
Yamagata University, Yamagata 990--8560, Japan, and
PRESTO, Japan Science and Technology Agency, Kawaguchi,
Saitama 332--0012, Japan. email: mharada@sci.kj.yamagata-u.ac.jp}}

\maketitle

\begin{abstract}
In this paper, we determine the dimensions of
the residue codes of extremal Type~II $\ZZ_4$-codes
for lengths $32$ and $40$.
We demonstrate that 
every binary doubly even self-dual code
of length $32$ can be realized as the residue code of some 
extremal Type~II $\ZZ_4$-code.
It is also shown that
there is a unique extremal Type~II $\ZZ_4$-code of
length $32$ whose residue code has the smallest 
dimension $6$ up to equivalence.
As a consequence, many new extremal Type~II $\ZZ_4$-codes 
of lengths $32$ and $40$ are constructed.
\end{abstract}



\section{Introduction}\label{Sec:1}

As described in~\cite{RS-Handbook},
self-dual codes are an important class of linear codes for both
theoretical and practical reasons.
It is a fundamental problem to classify self-dual codes
of modest length, 
and construct self-dual codes with the largest minimum weight
among self-dual codes of that length.
Among self-dual $\ZZ_k$-codes, self-dual $\ZZ_4$-codes
have been widely studied because such codes
have applications to unimodular lattices and 
nonlinear binary codes,
where $\ZZ_k$ denotes the ring of integers modulo $k$
and $k$ is a positive integer.

A $\ZZ_4$-code $C$ is {Type~II}
if $C$ is self-dual and 
the Euclidean weights of all codewords of $C$ are divisible by 
8~\cite{Z4-BSBM,Z4-HSG}.
This is a remarkable class of self-dual $\ZZ_4$-codes
related to even unimodular lattices.
A Type~II $\ZZ_4$-code of length $n$ exists if and 
only if $n \equiv 0 \pmod 8$,
and the minimum Euclidean 
weight $d_E$ of a Type~II $\ZZ_4$-code  
of length $n$ is bounded by
$d_E \le 8 \lfloor n/24 \rfloor +8$~\cite{Z4-BSBM}.
A Type~II $\ZZ_4$-code meeting this bound with equality is called 
{extremal}.
If $C$ is a Type~II $\ZZ_4$-code, then 
the residue code $ C^{(1)}$ is a binary 
doubly even code containing the all-ones vector 
$\allone$~\cite{Z4-CS, Z4-HSG}.

It follows from the mass formula in~\cite{G} that
for a given binary doubly even code $B$ containing $\allone$
there is a Type~II $\ZZ_4$-code $C$ with $C^{(1)}=B$.
However, it is not known in general whether there is an extremal
Type~II $\ZZ_4$-code $C$ with $C^{(1)}=B$ or not.
Recently, at length $24$, 
binary doubly even codes which are
the residue codes of extremal Type~II $\ZZ_4$-codes
have been classified in~\cite{HLM}.
In particular, 
there is an extremal Type~II $\ZZ_4$-code
whose residue code has dimension $k$
if and only if $k \in \{6,7,\ldots,12\}$~\cite[Table 1]{HLM}.
It is shown that
there is a unique extremal Type~II $\ZZ_4$-code
with residue code of dimension $6$ up to equivalence~\cite{HLM}.
Also, every binary doubly even self-dual code
of length $24$ can be realized as the residue code of some 
extremal Type~II $\ZZ_4$-code~\cite[Postscript]{CS97} 
(see also~\cite{HLM}).
Since extremal Type~II $\ZZ_4$-codes of length $24$ and their
residue codes are related to the Leech lattice~\cite{Z4-BSBM,CS97}
and structure codes of the moonshine vertex operator 
algebra~\cite{HLM}, respectively, 
this length is of special interest.
For the next two lengths $n=32$ and $40$, 
a number of extremal Type~II $\ZZ_4$-codes are known 
(see~\cite{Huffman05}).
However, 
only a few extremal Type~II $\ZZ_4$-codes which have
residue codes of dimension less than $n/2$
are known for these lengths $n$.
This motivates us to study the dimensions of
the residue codes of extremal Type~II $\ZZ_4$-codes for
these lengths.

In this paper, it is shown that
there is an extremal Type~II $\ZZ_4$-code of length $32$
whose residue code has dimension $k$
if and only if $k \in \{6,7,\ldots,16\}$.
In particular, we study two cases $k=6$ and $16$.
We demonstrate that 
every binary doubly even self-dual code
of length $32$ can be realized as the residue code of some 
extremal Type~II $\ZZ_4$-code.
It is also shown that
there is a unique extremal Type~II $\ZZ_4$-code of
length $32$ with residue code of dimension $6$ up to equivalence.
Finally, it is shown that
there is an extremal Type~II $\ZZ_4$-code of length $40$
whose residue code has dimension $k$
if and only if $k \in \{7,8,\ldots,20\}$.
As a consequence, many new extremal Type~II $\ZZ_4$-codes 
of lengths $32$ and $40$ are constructed.
Extremal Type~II $\ZZ_4$-codes of lengths $32$ and $40$
are used to
construct extremal even unimodular lattices 
by Construction A (see~\cite{Z4-BSBM}).
All computer calculations in this paper
were done by {\sc Magma}~\cite{Magma}.


 \section{Preliminaries}\label{Sec:2}

\subsection{Extremal Type~II $\ZZ_4$-codes}

Let $\ZZ_4\ (=\{0,1,2,3\})$ denote the ring of integers
modulo $4$.
A {\em $\ZZ_{4}$-code} $C$ of length $n$ 
is a $\ZZ_{4}$-submodule of $\ZZ_{4}^n$.
Two $\ZZ_4$-codes are {\em equivalent} if one can be obtained from the
other by permuting the coordinates and (if necessary) changing
the signs of certain coordinates.
The {\em dual} code $C^\perp$ of $C$ is defined as
$C^\perp = \{ x \in \ZZ_{4}^n\ | \ x \cdot y = 0$ for all $y \in C\}$,
where $x \cdot y = x_1 y_1 + \cdots + x_n y_n$ for
$x=(x_1,\ldots,x_n)$ and $y=(y_1,\ldots,y_n)$.
A code $C$ is {\em self-dual} if $C=C^\perp.$

The {\em Euclidean weight} of a codeword $x=(x_1,\ldots,x_n)$ of $C$ is
$n_1(x)+4n_2(x)+n_3(x)$, where $n_{\alpha}(x)$ denotes
the number of components $i$ with $x_i=\alpha$ $(\alpha=1,2,3)$.
The {\em minimum Euclidean weight} $d_E$ of $C$ is the smallest Euclidean
weight among all nonzero codewords of $C$.
A $\ZZ_4$-code $C$ is {\em Type~II}
if $C$ is self-dual and 
the Euclidean weights of all codewords of $C$ are divisible 
by 8~\cite{Z4-BSBM,Z4-HSG}.
A Type~II $\ZZ_4$-code of length $n$ exists if and 
only if $n \equiv 0 \pmod 8$, and 
the minimum Euclidean 
weight $d_E$ of a Type~II $\ZZ_4$-code  
of length $n$ is bounded by
$d_E \le 8 \lfloor n/24 \rfloor +8$~\cite{Z4-BSBM}.
A Type~II $\ZZ_4$-code meeting this bound with equality is called 
{\em extremal}.

The classification of 
Type~II $\ZZ_4$-codes has been done for lengths $8$
and $16$~\cite{Z4-CS,Z4-PLF}.
At lengths $24,32$ and $40$, a number of extremal
Type~II $\ZZ_4$-codes are known (see~\cite{Huffman05}).
At length $48$, only two inequivalent extremal Type~II
$\ZZ_4$-codes are known~\cite{Z4-BSBM,Venkov}.
At lengths $56$ and $64$, recently, 
an extremal Type~II $\ZZ_4$-code has been constructed in~\cite{H10}.

\subsection{Binary doubly even self-dual codes}
Throughout this paper,
we denote by $\dim(B)$ the dimension of a binary code $B$.
Also, for a binary code $B$ and a binary vector $v$, 
we denote by $\langle B,v\rangle$ 
the binary code generated by the
codewords of $B$ and $v$.
A binary code $B$ is called doubly even if 
$\wt(x) \equiv 0 \pmod 4$ for any codeword $x \in B$,
where $\wt(x)$ denotes the weight of $x$.
A binary doubly even self-dual code of length $n$ exists if and 
only if $n \equiv 0 \pmod 8$, and 
the minimum weight 
$d$ of a binary doubly even self-dual code of length $n$
is bounded by $d \le 4 \lfloor n/24 \rfloor +4$ 
(see~\cite{Huffman05,RS-Handbook}).
A binary doubly even self-dual code
meeting this bound with equality is called  extremal.

Two binary codes $B$ and $B'$ are equivalent, denoted $B \cong B'$,
if $B$ can be obtained from $B'$ by permuting the coordinates.
The classification of binary doubly even self-dual codes
has been done for lengths up to $32$ 
(see~\cite{CPS,Huffman05,RS-Handbook}).
There are $85$ inequivalent binary
doubly even self-dual codes of length $32$,
five of which are extremal~\cite{CPS}.

\subsection{Residue codes of $\ZZ_4$-codes}
Every $\ZZ_4$-code $C$ of length $n$ has two binary codes 
$C^{(1)}$ and $C^{(2)}$ associated with $C$:
\[
C^{(1)}= \{ c \bmod 2 \mid  c \in C \} \text{  and }
C^{(2)}= \left\{ c \bmod 2 \mid c \in \ZZ_4^n, 2c\in C \right\}.
\]
The binary codes $C^{(1)}$ and $C^{(2)}$ are called the 
{\em residue} and {\em torsion} codes of $C$, respectively.
If $C$ is self-dual, then $ C^{(1)}$ is a binary 
doubly even code with $C^{(2)} = {C^{(1)}}^{\perp}$~\cite{Z4-CS}.
If $C$ is Type~II, then $C^{(1)}$ contains
the all-ones vector $\allone$~\cite{Z4-HSG}.

The following two lemmas can be easily shown
(see~\cite{HLM} for length $24$).

\begin{lem}
Let $B$ be the residue code of an extremal Type~II $\ZZ_4$-code
of length $n \in \{24,32,40\}$.
Then $B$ satisfies the following conditions:
\begin{align}
&B\text{ is doubly even;}\label{eq:b1}\\
&B\ni\allone;\label{eq:b2}\\
&B^\perp\text{ has minimum weight at least $4$}.\label{eq:b3}
\end{align}
\end{lem}
\begin{proof}
The assertions (\ref{eq:b1}) and (\ref{eq:b2}) follow 
from~\cite{Z4-CS} and~\cite{Z4-HSG}, respectively, as described above.
If $C$ is an extremal Type~II $\ZZ_4$-code of 
length $n$, then ${C^{(2)}}$ has minimum weight 
at least $2\lfloor n/24 \rfloor+2$ (see~\cite{H10}).
The assertion (\ref{eq:b3}) follows.
\end{proof}

\begin{lem}\label{Lem:bound}
Let $B$ be the residue code of an extremal Type~II $\ZZ_4$-code
of length $n$.
Then, 
$6 \le \dim(B) \le 16  \text{ if } n=32$,
and 
$7 \le \dim(B) \le 20  \text{ if } n=40$. 
\end{lem}
\begin{proof}
Since a binary doubly even code is self-orthogonal,
$\dim(B) \le n/2$.
From (\ref{eq:b3}), $B^\perp$ has minimum weight 
at least $4$.
It is known that a $[32,k,4]$ code exists only if $k \le 26$
and a $[40,k,4]$ code exists only if $k \le 33$
(see~\cite{Brouwer-Handbook}). 
The result follows.
\end{proof}

In this paper, we consider the existence of an extremal Type~II 
$\ZZ_4$-code with residue code of dimension $k$
for a given $k$.
To do this, the following lemma is useful, and it was shown 
in~\cite{HLM} for length $24$.
Since its modification to lengths $32$ and $40$ is straightforward,
we omit to give a proof.

\begin{lem}\label{Lem:wt4}
Let $C$ be an extremal Type~II $\ZZ_4$-code of length 
$n \in \{24, 32,40\}$.
Let $v$ be a binary
vector of length $n$ and weight $4$ such that
$v \not\in C^{(1)}$ and 
the code $\langle C^{(1)},v\rangle$ is doubly even.
Then there is an extremal Type~II $\ZZ_4$-code
$C'$ such that $C'^{(1)}=\langle C^{(1)},v\rangle$.
\end{lem}

\subsection{Construction method}\label{SS:method}
In this subsection, we review the method of construction of
Type~II $\ZZ_4$-codes, which was given in~\cite{Z4-PLF}.
Let $C_1$ be a binary code of length
$n \equiv 0 \pmod 8$ and dimension $k$ 
satisfying the conditions (\ref{eq:b1}) and (\ref{eq:b2}).
Without loss of generality, we may assume that $C_1$ has
generator matrix of the following form:
\begin{equation}  \label{Eq:G1}
G_1=
\left(\begin{array}{cc}
  A    &   \tilde{I_k} \\
\end{array}\right),
\end{equation}
where 
$A$ is a $k \times (n-k)$ matrix which has the property that
the first row is  $\allone$,
${\displaystyle
\tilde{I_k}=
\left(\begin{array}{cccccc}
1       & \cdots   & 1 \\
0       &          &   \\
\vdots  &   I_{k-1}&   \\
0       &          &   \\
\end{array}\right),
}$
and
$I_{k-1}$ denotes the identity matrix of order $k-1$.
Since $C_1$ is self-orthogonal,
the matrix $G_1$ can be extended to a generator matrix 
of $C_1^\perp$ as follows
$
\left(\begin{array}{cc}
G_1 \\
D
\end{array}\right).
$
Then there are $2^{1+k(k-1)/2}$
$k \times k$ $(1,0)$-matrices $B$ such that the following matrices
\begin{equation}\label{Eq:Z4-G}
\left(\begin{array}{cccccc}
&A &    & \tilde{I_k}+2B \\
&  & 2D &               \\
\end{array}\right)
\end{equation}
are generator matrices of Type~II $\ZZ_4$-codes $C$,
where we regard the matrices as matrices over $\ZZ_4$.
That is, there are $2^{1+k(k-1)/2}$ Type~II $\ZZ_4$-codes
$C$ with $C^{(1)}=C_1$~\cite{G,Z4-PLF}.

Since any Type~II $\ZZ_4$-code is equivalent to some 
Type~II $\ZZ_4$-code
containing $\allone$~\cite{Z4-HSG},
without loss of generality, we may assume that the first row
of $B$ is the zero vector.
This reduces 
our search space for finding extremal Type~II $\ZZ_4$-codes.
In fact, there are only $2^{(k-1)(k-2)/2}$ Type~II $\ZZ_4$-codes 
$C$ with $C^{(1)}=C_1$ containing $\allone$
(see also~\cite{BM09}).

\section{Extremal Type~II $\ZZ_4$-codes of length 32}

\subsection{Known extremal Type~II $\ZZ_4$-codes of length 32}
Currently,
$57$ inequivalent
extremal Type~II $\ZZ_4$-codes of length $32$ are known 
(see~\cite{Z4-GaH,Huffman05}).
Among the $57$ known codes,
$54$ codes have residue codes which are
extremal doubly even self-dual
codes.
In particular, 
for every binary extremal doubly even self-dual
code $B$ of length $32$, 
there is an extremal Type~II $\ZZ_4$-code $C$
with $C^{(1)} \cong B$~\cite{Z4-GaH}. 

Only $C_{5,1}$ in~\cite{Z4-BSBM} and
$\tilde C_{31,2}, \tilde C_{31,3}$ in~\cite{Z4-PSQ}
are known  extremal Type~II $\ZZ_4$-codes 
whose residue codes are not extremal doubly even self-dual
codes (see~\cite{Z4-GaH}).
The residue codes of  $\tilde C_{31,2}, \tilde C_{31,3}$ 
in~\cite{Z4-PSQ} have dimension $11$.
The residue code of $C_{5,1}$ in~\cite{Z4-BSBM}
is the first order Reed--Muller code
$RM(1,5)$ of length $32$, thus, $\dim(C_{5,1}^{(1)})=6$.
In Section~\ref{Subsec:32-dim6}, 
we show that
there is a unique extremal Type~II $\ZZ_4$-code of
length $32$ with residue code of dimension $6$, up to equivalence.

\subsection{Determination of dimensions of residue codes}

By Lemma~\ref{Lem:bound}, 
if $C$ is an extremal Type~II $\ZZ_4$-code of length 
$32$, then $6 \le \dim(C^{(1)}) \le 16$. 
In this subsection, we show the converse assertion using Lemma~\ref{Lem:wt4}.
To do this, we first
fix the coordinates of $RM(1,5)$ 
by considering the following matrix  as a 
generator matrix of $RM(1,5)$:
\begin{equation}\label{eq:RM}
\left(\begin{array}{cccc}
1111 1111& 1111 1111& 1111 1111& 1111 1111\\
1111 1111& 1111 1111& 0000 0000& 0000 0000\\
1111 1111& 0000 0000& 1111 1111& 0000 0000\\
1111 0000& 1111 0000& 1111 0000& 1111 0000\\
1100 1100& 1100 1100& 1100 1100& 1100 1100\\
1010 1010& 1010 1010& 1010 1010& 1010 1010
\end{array}\right).
\end{equation}
It is well known that $RM(1,5)$ has 
the following weight enumerator:
\begin{equation}\label{eq:WE}
1 + 62 y^{16} + y^{32}.
\end{equation}

For $i=7,8,\ldots,15$, 
we define $B_{32,i}$ to be the binary code
$\langle B_{32,i-1}, v_i \rangle$,
where $B_{32,6}=RM(1,5)$ and the support $\supp(v_i)$ of
the vector $v_i$ is listed in Table~\ref{Tab:32}.
The weight distributions of $B_{32,i}$ $(i=7,8,\ldots,15)$
are also listed in the table, where
$A_j$ denotes the number of codewords of weight $j$
($j=4,8,12,16$).
From the weight distributions,
one can easily verify that $v_i \not\in B_{32,i-1}$ and 
$B_{32,i}$ is doubly even
for $i=7,8,\ldots,15$.
Note that the code $C_{5,1}$ in~\cite{Z4-BSBM} is an
extremal Type~II $\ZZ_4$-code with residue code $RM(1,5)$, 
and there are extremal Type~II $\ZZ_4$-codes
with residue codes of dimension $16$.
By Lemma~\ref{Lem:wt4}, 
we have the following:

\begin{prop}\label{Prop:1}
There is an extremal Type~II $\ZZ_4$-code of length $32$
whose residue code has dimension $k$
if and only if $k \in \{6,7,\ldots,16\}$.
\end{prop}

\begin{rem}
In the next two subsections,
we study two cases $k=6$ and $16$.
\end{rem}

\begin{table}[thb]
\centering
\caption{Supports $\supp(v_i)$ and weight distributions of $B_{32,i}$}
\vspace*{0.2in}
\label{Tab:32}
{\small
\begin{tabular}{c|l|rrrr}
\noalign{\hrule height1pt}
$i$ & \multicolumn{1}{c|}{$\supp(v_i)$} 
&$A_4$ &$A_8$ & $A_{12}$& $A_{16}$ \\
\hline
 7&$\{1, 2,  3,  4\}$ &  1&  0&   7&   110\\ 
 8&$\{1, 2,  5,  6\}$ &  3&  0&  21&   206\\ 
 9&$\{1, 2,  7,  8\}$ &  6&  4&  42&   406\\ 
10&$\{1, 2,  9, 10\}$ & 10& 12& 102&   774\\ 
11&$\{1, 2, 11, 12\}$ & 16& 36& 208&  1526\\ 
12&$\{1, 2, 13, 14\}$ & 28& 84& 420&  3030\\
13&$\{1, 2, 17, 18\}$ & 36&196& 924&  5878\\
14&$\{1, 2, 19, 20\}$ & 48&428&1936& 11558\\
15&$\{1, 2, 21, 22\}$ & 72&892&3960& 22918\\
\noalign{\hrule height1pt}
\end{tabular}
}
\end{table}

As another approach to Proposition~\ref{Prop:1},
we explicitly found an extremal Type~II $\ZZ_4$-code
$C_{32,i}$ with $C_{32,i}^{(1)} \cong B_{32,i}$
for $i=7,8,\ldots,15$,
using the method given in Section~\ref{SS:method}.
Any $\ZZ_4$-code with residue code of dimension $k$
is equivalent to a code with generator
matrix of the form:
\begin{equation}  \label{Eq:SF}
\left(\begin{array}{cc}
  I_k   &   A \\
  O     &2B
\end{array}\right),
\end{equation}
where $A$ is a matrix over $\ZZ_4$ and $B$ is
a $(1,0)$-matrix.
For these codes $C_{32,i}$, we give generator matrices
of the form (\ref{Eq:SF}), by only listing in Figure~\ref{Fig:dim} 
the $i \times (32-i)$ matrices $A$ in (\ref{Eq:SF}) to save space.
Note that the lower part in (\ref{Eq:SF})
can be obtained from the matrices
$\left(\begin{array}{cc}
  I_k   &   A \\
\end{array}\right)$,
since $C^{(2)} = {C^{(1)}}^{\perp}$ and  
$\left(\begin{array}{cc}
  I_k   &   A \bmod 2\\
\end{array}\right)$ is a generator matrix of
${C^{(1)}}$,
where $A \bmod 2$ denotes the binary matrix
whose $(i,j)$-entry is $a_{ij} \bmod 2$ for $A=(a_{ij})$.

\begin{figure}[thb]
\centering
{\footnotesize
\begin{tabular}{ll}
$
\left(\begin{array}{c}
0 0 0 0 0 0 0 0 0 0 0 0 0 0 0 0 0 2 0 0 3 3 3 2 2\\
1 1 0 0 1 1 1 1 0 0 1 1 0 0 0 1 1 2 1 1 1 1 0 1 2\\
1 0 1 0 1 0 1 0 1 0 1 0 1 0 1 1 0 1 1 0 1 0 1 1 1\\
0 1 1 0 0 1 1 0 0 1 1 0 0 1 1 0 1 1 0 1 0 1 1 1 1\\
0 0 1 1 1 1 0 0 0 0 1 1 1 1 0 0 0 0 1 1 2 2 2 0 3\\
0 0 1 1 1 1 1 1 1 1 0 0 0 0 0 1 1 3 0 0 2 2 2 0 0\\
1 1 1 1 1 1 1 1 1 1 1 1 1 1 1 0 0 0 0 0 0 0 0 0 0
\end{array}\right)
$&
$
\left(\begin{array}{c}
0 0 0 0 0 0 0 0 0 0 0 0 0 0 0 0 0 3 3 3 2 2 2 0\\
0 0 0 0 0 0 0 0 0 0 0 0 0 0 0 0 0 1 3 2 1 2 0 2\\
1 0 1 0 1 0 1 0 1 0 1 0 1 0 1 1 0 1 0 1 1 1 1 1\\
1 0 1 0 1 0 0 1 0 1 0 1 0 1 1 1 0 0 3 3 1 3 3 1\\
1 1 1 1 0 0 1 1 0 0 0 0 1 1 0 1 1 0 2 0 0 3 0 0\\
1 1 0 0 1 1 1 1 0 0 1 1 0 0 0 1 1 0 2 0 2 2 3 0\\
0 0 1 1 1 1 1 1 1 1 0 0 0 0 0 1 1 2 2 0 2 2 2 3\\
1 1 1 1 1 1 1 1 1 1 1 1 1 1 1 0 0 0 0 0 0 0 0 0
\end{array}\right)
$\\$
\left(\begin{array}{c}
0 0 0 0 0 0 2 0 0 0 0 0 0 2 0 0 3 3 0 0 0 3 2\\
0 0 0 0 0 0 0 0 0 0 0 0 0 0 0 0 1 3 2 2 0 2 1\\
1 0 1 0 1 0 1 1 0 1 0 1 0 1 1 0 1 0 1 1 1 1 1\\
1 0 1 0 1 0 1 1 0 1 0 1 0 3 1 0 0 3 1 3 1 3 1\\
1 1 1 1 0 0 2 1 1 1 1 0 0 2 1 1 2 2 2 0 3 0 0\\
1 1 0 0 1 1 2 1 1 0 0 1 1 0 1 1 2 2 3 2 0 0 2\\
0 0 1 1 1 1 2 0 0 1 1 1 1 2 1 1 0 2 0 3 2 2 0\\
0 0 0 0 0 0 0 1 1 1 1 1 1 1 0 0 2 0 0 2 2 0 2\\
1 1 1 1 1 1 1 0 0 0 0 0 0 0 0 0 0 0 2 2 2 2 0
\end{array}\right)
$&
$
\left(\begin{array}{c}
0 0 0 0 0 0 2 0 0 0 0 0 0 3 3 3 2 2 2 0 0 2\\
0 0 0 0 0 0 2 0 0 0 0 0 0 1 3 2 1 2 0 0 0 0\\
0 0 0 0 0 0 2 0 0 0 0 0 0 1 3 2 0 1 2 0 2 2\\
0 0 0 0 0 0 2 0 0 0 0 0 0 1 3 2 2 0 1 2 2 0\\
0 0 1 1 1 1 0 0 0 1 1 1 1 3 3 2 0 0 2 3 2 0\\
1 1 0 0 1 1 2 1 1 0 0 1 1 3 3 2 0 2 2 0 3 2\\
1 0 1 0 1 0 1 1 0 1 0 1 0 1 0 1 1 1 1 1 1 1\\
0 1 0 1 1 0 3 0 1 0 1 1 0 0 3 3 1 1 1 3 3 3\\
0 0 0 0 0 0 0 1 1 1 1 1 1 2 0 0 2 0 2 2 0 1\\
1 1 1 1 1 1 1 0 0 0 0 0 0 2 0 0 0 0 0 2 0 0
\end{array}\right)
$\\$
\left(\begin{array}{c}
0 0 0 0 0 0 2 0 0 0 0 0 1 3 1 2 0 2 0 0 0\\
0 0 0 0 0 0 2 0 0 2 0 0 1 3 0 1 0 2 2 0 2\\
0 0 0 0 0 0 0 0 0 2 0 0 3 3 0 0 3 2 0 0 2\\
0 0 0 0 0 0 0 0 0 2 0 0 1 3 0 0 2 1 2 0 0\\
0 0 0 0 0 0 2 0 0 2 0 0 1 3 2 0 0 0 1 0 0\\
1 1 0 0 1 1 2 1 1 0 1 1 3 3 0 2 0 0 0 2 3\\
1 0 1 0 1 0 1 1 0 1 1 0 1 0 1 1 1 1 1 1 1\\
0 1 1 0 0 1 3 0 1 1 0 1 0 3 1 1 3 1 1 1 3\\
0 0 1 1 1 1 0 0 0 2 1 1 0 0 0 0 0 0 0 1 2\\
0 0 1 1 1 1 2 1 1 1 0 0 0 0 0 2 2 2 2 2 0\\
1 1 1 1 1 1 1 0 0 2 0 0 2 0 0 0 2 2 0 0 0
\end{array}\right)
$&
$
\left(\begin{array}{c}
0 0 0 0 0 0 2 0 0 3 3 2 0 2 0 0 2 2 3 0\\
0 0 0 0 0 0 0 0 0 3 3 0 0 2 0 2 2 2 2 3\\
0 0 0 0 0 0 2 0 0 1 3 2 2 0 2 0 2 1 2 2\\
0 0 0 0 0 0 2 0 0 1 3 0 2 2 0 0 1 0 2 2\\
0 0 0 0 0 0 2 0 0 3 3 0 2 0 0 3 0 0 2 0\\
0 0 0 0 0 0 2 0 0 3 3 0 0 0 3 2 0 2 2 2\\
1 0 1 0 1 0 1 1 0 1 0 1 1 1 1 1 1 1 1 1\\
1 0 1 0 1 0 1 1 0 0 3 1 1 3 3 3 1 1 3 3\\
1 1 1 1 0 0 2 1 1 2 0 0 2 1 2 2 0 2 0 0\\
1 1 0 0 1 1 0 1 1 0 0 0 1 2 0 2 2 2 0 0\\
0 0 1 1 1 1 0 1 1 0 0 1 0 0 0 0 0 2 2 0\\
1 1 1 1 1 1 1 0 0 0 0 0 0 2 2 2 2 2 2 0
\end{array}\right)
$
\end{tabular}
\caption{
Matrices $A$
in generator matrices of $C_{32,i}$} 
\label{Fig:dim}
}
\end{figure}
\begin{figure}[thb]
\centering
{\footnotesize
\begin{tabular}{ll}
$
\left(\begin{array}{c}
0 0 0 0 0 0 1 1 3 0 2 2 0 0 2 2 2 2 2\\
0 0 0 0 0 0 3 1 0 1 2 0 0 2 0 2 0 2 2\\
0 0 0 0 0 0 3 3 0 2 3 0 2 2 0 0 0 2 2\\
0 0 0 0 0 0 1 3 2 2 0 1 2 0 0 0 0 2 2\\
0 0 0 0 0 0 3 3 2 0 0 2 3 0 0 2 0 0 0\\
0 0 0 0 0 0 3 3 2 2 0 0 2 3 2 2 0 0 0\\
0 0 0 0 0 0 1 3 2 2 0 2 0 2 1 2 2 0 2\\
0 0 0 0 0 0 1 1 0 2 2 0 0 0 2 3 0 2 2\\
0 0 1 1 1 1 3 3 0 0 2 0 2 2 2 2 1 0 2\\
1 1 0 0 1 1 1 3 2 0 2 0 0 0 2 2 2 3 2\\
1 0 1 0 1 0 1 0 3 3 1 1 1 1 1 3 3 3 3\\
0 1 0 1 1 0 0 1 3 1 1 3 1 1 3 3 3 1 3\\
1 1 1 1 1 1 2 2 2 2 2 2 0 0 2 2 2 2 1
\end{array}\right)
$
$
\left(\begin{array}{c}
0 0 0 0 0 1 1 3 0 2 0 0 2 2 2 0 0 2\\
0 0 0 0 0 1 1 2 3 0 0 2 2 0 2 2 0 0\\
0 0 2 0 0 1 3 2 0 1 0 0 2 2 0 2 0 0\\
0 0 0 0 0 3 1 0 0 2 1 2 0 2 0 0 2 2\\
0 0 0 0 0 1 3 0 2 2 0 1 2 0 2 0 2 0\\
0 0 0 0 0 1 3 2 2 0 2 0 1 0 2 2 0 0\\
0 0 0 0 0 3 1 2 0 0 0 2 2 1 0 0 0 0\\
0 0 0 0 0 3 1 2 2 2 2 0 0 2 1 0 2 2\\
0 0 2 0 0 3 3 2 0 2 0 0 2 0 0 3 2 0\\
1 1 0 1 1 1 3 2 0 2 0 2 2 2 0 0 2 3\\
1 0 1 1 0 3 0 1 1 3 1 3 3 1 1 3 3 1\\
0 1 1 0 1 0 3 1 1 1 3 1 1 3 3 3 1 3\\
0 0 0 1 1 2 0 2 2 2 0 0 2 2 0 0 3 2\\
1 1 3 0 0 2 2 0 0 2 0 0 0 0 0 2 0 2
\end{array}\right)
$
$
\left(\begin{array}{c}
0 0 1 1 0 2 2 2 0 2 2 2 0 0 0 2 3\\
0 0 1 3 0 0 0 2 2 0 2 0 2 0 0 1 2\\
0 0 3 3 2 2 0 2 2 2 2 0 0 0 3 0 2\\
0 0 3 1 2 2 0 2 2 2 2 0 0 1 0 2 0\\
0 0 1 1 0 0 2 3 0 0 0 2 2 2 0 2 0\\
0 0 1 1 0 2 0 2 3 0 2 2 0 2 0 2 2\\
0 0 3 3 0 2 2 2 2 3 2 2 0 2 0 0 0\\
0 0 3 1 0 0 0 0 2 2 1 0 2 0 2 0 2\\
0 0 1 1 2 2 0 0 0 0 0 3 0 0 2 0 0\\
0 0 3 1 0 2 0 2 0 0 0 0 1 2 0 0 0\\
1 0 3 0 1 3 3 1 1 3 1 1 1 1 3 3 1\\
1 0 0 1 3 3 3 3 3 1 1 3 1 1 1 3 3\\
1 1 0 0 0 0 3 2 0 2 0 0 0 0 0 0 2\\
1 1 0 0 2 3 2 0 2 2 0 2 2 2 2 0 2\\
1 1 0 2 3 0 2 2 2 2 2 0 2 0 0 2 2
\end{array}\right)
$
\end{tabular}
\setcounter{figure}{0}
\caption{
Matrices $A$ in generator matrices of $C_{32,i}$
(continued)} 
}
\end{figure}

\subsection{Residue codes of dimension 16}

As described above,
there are $85$ inequivalent binary
doubly even self-dual codes of length $32$.
These codes are denoted by 
$\text{C1}, \text{C2}, \ldots, \text{C85}$
in~\cite[Table A]{CPS}, where 
$\text{C81}, \ldots, \text{C85}$ are extremal.
For each $B$ of the $5$ extremal ones,
there is an extremal Type~II $\ZZ_4$-code $C$
with $C^{(1)} \cong B$~\cite{Z4-GaH}.

Using the method given in Section~\ref{SS:method}, 
we explicitly found an extremal Type~II $\ZZ_4$-code $D_{32,i}$
with $D_{32,i}^{(1)} \cong \text{C}i$ for $i=1,2,\ldots,80$.
The matrices $M_{i}$ in 
generator matrices  of the form
$\left(\begin{array}{cc}
  I_{16}   &   M_i \\
\end{array}\right)$
$(i=1,2,\ldots,80)$
can be obtained electronically from
\begin{verbatim}
http://sci.kj.yamagata-u.ac.jp/~mharada/Paper/z4-32.txt
\end{verbatim}
Hence, we have the following:
\begin{prop}\label{Prop:32dim16}
Every binary doubly even self-dual code
of length $32$ can be realized as the residue code of some 
extremal Type~II $\ZZ_4$-code.
\end{prop}


Among known $57$ inequivalent
extremal Type~II $\ZZ_4$-codes of length $32$,
the residue codes of $54$ codes are
extremal doubly even self-dual
codes and the residue codes of the other three codes 
$C_{5,1}$ in~\cite{Z4-BSBM} and
$\tilde C_{31,2}, \tilde C_{31,3}$ in~\cite{Z4-PSQ}
have dimensions $6,11$ and $11$, respectively.
In particular, 
$\tilde C_{31,2}^{(1)}$ and $\tilde C_{31,3}^{(1)}$
have the following identical weight enumerators:
\[
1+ 496^{12}+1054^{16}+ 496^{20}+y^{32}.
\]
Hence, none of $\tilde C_{31,2}$ and $\tilde C_{31,3}$
is equivalent to $C_{32,11}$.
The code $C_{32,i}^{(1)}$ 
has dimension $i$ for $i=7,8,\ldots,15$, 
and $D_{32,i}^{(1)}$ is a 
non-extremal doubly even self-dual code
for $i=1,2,\ldots,80$. 
Since 
equivalent $\ZZ_4$-codes have equivalent residue codes,
we have the following:

\begin{cor}
There are at least $146$ inequivalent
extremal Type~II $\ZZ_4$-codes of length $32$.
\end{cor}

\begin{rem}\label{rem:32}
The torsion code of any of the $9$ codes 
$C_{32,i}$ $(i=7,8,\ldots,15)$
has minimum weight $4$, since 
the residue code has minimum weight $4$
and the torsion code of an extremal Type~II $\ZZ_4$-code
contains no codeword of weight $2$.
The torsion code of any of the $80$ codes 
$D_{32,i}$ $(i=1,2,\ldots,80)$
has minimum weight $4$.
By Theorem 1 in \cite{Rains}, any of the $89$ codes 
$C_{32,i}$ and $D_{32,i}$ 
has minimum Hamming weight $4$.
In addition, 
any of the codes has minimum Lee weight $8$,
since the minimum Lee weight of an extremal Type~II $\ZZ_4$-code
with minimum Hamming weight $4$ is $8$
(see~\cite{Z4-BSBM} for the definitions).
\end{rem}


\subsection{Residue codes of dimension 6}
\label{Subsec:32-dim6}

At length $24$,
the smallest dimension among codes 
satisfying the conditions (\ref{eq:b1})--(\ref{eq:b3}) is $6$.
There is a unique binary $[24,6]$ code 
satisfying (\ref{eq:b1})--(\ref{eq:b3}),
and there is a unique extremal Type~II
$\ZZ_4$-code with residue code of dimension $6$ 
up to equivalence~\cite{HLM}.
In this subsection, we show that a similar situation holds 
for length $32$.

\begin{lem}\label{Lem:RM}
Up to equivalence, $RM(1,5)$ is a unique binary
$[32,6]$ code
satisfying the conditions (\ref{eq:b1})--(\ref{eq:b3}).
\end{lem}
\begin{proof}
Let $B_{32}$ be a binary $[32,6]$ code
satisfying (\ref{eq:b1})--(\ref{eq:b3}).
From (\ref{eq:b1}) and (\ref{eq:b2}),
the weight enumerator of $B_{32}$ is written as:
\[
1 + ay^4 
+ by^8 
+ cy^{12} 
+ (62 -2a - 2b - 2c)y^{16}
+ cy^{20} 
+ by^{24} 
+ ay^{28} 
+ y^{32},
\]
where $a,b$ and $c$ are nonnegative integers.
By the MacWilliams identity, 
the weight enumerator of $B_{32}^\perp$
is given by:
\[
1 + (9a + 4b + c)y^2 + (294a + 24b - 10c + 1240)y^4  + \cdots.
\]
From (\ref{eq:b3}), $9a + 4b + c=0$.
This gives $a=b=c=0$, since all $a,b$ and $c$ are nonnegative.
Hence, the weight enumerator of $B_{32}$ is uniquely determined
as (\ref{eq:WE}).

Let $G$ be a generator matrix of $B_{32}$ and let
$r_i$ be the $i$th row of $G$ $(i=1,2,\ldots,6)$.
From the weight enumerator (\ref{eq:WE}), 
we may assume without loss of generality that
the first three rows of $G$ are as follows:
\[
\begin{array}{cccccc}
r_1 &=&(1111 1111& 1111 1111& 1111 1111& 1111 1111),\\
r_2 &=&(1111 1111& 1111 1111& 0000 0000& 0000 0000),\\
r_3 &=&(1111 1111& 0000 0000& 1111 1111& 0000 0000).
\end{array}
\]

Put $r_4=(v_1,v_2,v_3,v_4)$,
where $v_i$ $(i=1,2,3,4)$ are vectors of length $8$
and let $n_i$ denote the number of $1$'s in $v_i$.
Since the binary code $B_4$ 
generated by the four rows $r_1,r_2,r_3,r_4$ has
weight enumerator $1+14y^{16}+y^{32}$, we have
the following system of the equations:
\begin{align*}
\wt(r_4)&= n_1+n_2+n_3+n_4 = 16,\\
\wt(r_2+r_4)&= (8-n_1)+(8-n_2)+n_3+n_4 = 16,\\
\wt(r_3+r_4)&= (8-n_1)+n_2+(8-n_3)+n_4 = 16,\\
\wt(r_2+r_3+r_4)&=n_1+(8-n_2)+(8-n_3)+n_4 = 16.
\end{align*}
This system of the equations has a unique solution
$n_1=n_2=n_3=n_4=4$.
Hence, we may assume without loss of generality that
\[
\begin{array}{cccccc}
r_4 &=& (1111 0000& 1111 0000& 1111 0000& 1111 0000).
\end{array}
\]

Similarly, put $r_5=(v_1,v_2,\ldots,v_8)$,
where $v_i$ $(i=1,\ldots,8)$ are vectors of length $4$ 
and let $n_i$ denote the number of $1$'s in $v_i$.
Since the binary code $B_5=\langle B_4,r_5 \rangle$ has
weight enumerator $1+30y^{16}+y^{32}$,
we have the following system of the equations:
\[
\sum_{a \in \Gamma_t} n_a 
+ \sum_{b \in \{1,\ldots,8\}\setminus\Gamma_t} (4-n_b) = 16
\quad (t=1,\ldots,8),
\]
where $\Gamma_t$ $(t=1,\ldots,8)$ are
$\{1,\ldots,8\}$, 
$\{5, 6, 7, 8 \}$,
$\{3, 4, 7, 8 \}$,
$\{2, 4, 6, 8 \}$,
$\{1, 2, 7, 8 \}$,
$\{1, 3, 6, 8 \}$,
$\{1, 4, 5, 8 \}$ and 
$\{2, 3, 5, 8 \}$.
This system of the equations has a unique solution
$n_i=2\ (i=1,\ldots,8)$.
Hence, we may assume without loss of generality that
\[
\begin{array}{cccccc}
r_5 &=& (1100 1100& 1100 1100& 1100 1100& 1100 1100).
\end{array}
\]

Finally, put $r_6=(v_1,v_2,\ldots,v_{16})$, 
where $v_i$ $(i=1,\ldots,16)$ are vectors of length $2$ 
and let $n_i$ denote the number of $1$'s in $v_i$.
Similarly, since the binary code $\langle B_5,r_6 \rangle$ has
weight enumerator (\ref{eq:WE}),
we have
$n_i=1\ (i=1,\ldots,16)$.
Hence, we may assume without loss of generality that
\[
\begin{array}{cccccc}
r_6 &=& (1010 1010& 1010 1010& 1010 1010& 1010 1010).
\end{array}
\]
Therefore, 
a generator matrix $G$ is uniquely determined
up to permutation of columns.
\end{proof}

Using a classification method similar to that described 
in~\cite[Section 4.3]{HLM},
we verified that all Type~II $\ZZ_4$-codes
with residue codes $RM(1,5)$ are equivalent.
Therefore, we have the following:

\begin{prop}\label{Prop:32dim6}
Up to equivalence, there is a unique
extremal Type~II $\ZZ_4$-code of length $32$
with residue code of dimension $6$.
\end{prop}

By Proposition~\ref{Prop:32dim16} and Lemma~\ref{Lem:RM},
all binary $[32,k]$ codes satisfying (\ref{eq:b1})--(\ref{eq:b3}) 
can be realized as the residue codes of some 
extremal Type~II $\ZZ_4$-codes for $k=6$ and $16$.
The binary $[32,7]$ code $N_{32}=\langle RM(1,5),v\rangle$
satisfies (\ref{eq:b1})--(\ref{eq:b3}),
where $RM(1,5)$ is defined by (\ref{eq:RM}) and
\[
\supp(v)=\{ 1, 2, 3, 4, 5, 9, 17, 29\}.
\]
However, we verified that
none of the Type~II $\ZZ_4$-codes $C$ with 
$C^{(1)}=N_{32}$ is extremal,
using the method in Section~\ref{SS:method}.
Therefore, there is a binary code
satisfying (\ref{eq:b1})--(\ref{eq:b3}) 
which cannot be realized as the residue code of an 
extremal Type~II $\ZZ_4$-code of length $32$.

\section{Extremal Type~II $\ZZ_4$-codes of length 40}

\subsection{Determination of dimensions of residue codes}
Currently,
$23$ inequivalent
extremal Type~II $\ZZ_4$-codes of length $40$ are 
known~\cite{CS97,Z4-GaH, Z4-H, Z4-PSQ}.
Among these $23$ known codes,
the $22$ codes have residue codes which are 
doubly even self-dual codes
and the other code is given in~\cite{Z4-PSQ}.
Using an approach similar to that used in the previous
section, we determine the dimensions of the residue codes
of extremal Type~II $\ZZ_4$-codes of length $40$.

By Lemma~\ref{Lem:bound}, 
if $C$ is an extremal Type~II $\ZZ_4$-code of length 
$40$, then $7 \le \dim(C^{(1)}) \le 20$. 
Using the method given in Section~\ref{SS:method},
we explicitly found an extremal Type~II $\ZZ_4$-code 
from some binary doubly even $[40,7,16]$ code.
This binary code was found as a subcode of some
binary doubly even self-dual code.
We denote the extremal Type~II $\ZZ_4$-code by $C_{40,7}$.
The weight enumerators of 
$C_{40,7}^{(1)}$ and ${C_{40,7}^{(1)}}^\perp$ are given by:
\begin{align*}
&
1 +  15 y^{16}+ 96y^{20}+  15y^{24}+y^{40}, \\
&
1 + 1510 y^{4} + 59520 y^{6} + 1203885 y^{8} + 13235584 y^{10} 
+ 87323080 y^{12} 
\\ &
+ 362540160 y^{14} 
+ 982189650 y^{16} 
+ 1771386240 y^{18} 
+ 2154055332 y^{20}
\\ &
+ \cdots + y^{40},
\end{align*}
respectively.
For the code $C_{40,7}$,
we give a generator matrix of the form (\ref{Eq:Z4-G}),
by only listing the $7 \times 40$ matrix $G_{40}$
which has form $(\ A \ \ \tilde{I_7}+2B\ )$ in (\ref{Eq:Z4-G}):
\[
G_{40}=
\left(\begin{array}{cc}
111111111111111111111111111111111& 1111111\\
101101001011110000011001100000101& 0100000\\
100000101011011000100010001111011& 2210000\\
100110011011001101111111101000100& 0203000\\
011110110111111001011010010001010& 0002300\\
110100101111000011100110000010100& 0202010\\
010111101001111110010110110100010& 0002003
\end{array}\right).
\]
Note that the lower part in (\ref{Eq:Z4-G})
can be obtained from $G_{40}$.

Using the generator matrix $G_{40} \bmod 2$ of 
the binary code $C_{40,7}^{(1)}$,
we establish the existence of some extremal Type~II 
$\ZZ_4$-codes, by Lemma~\ref{Lem:wt4}, as follows.
For $i=8,9\ldots,19$, 
we define $B_{40,i}$ to be the binary code
$\langle B_{40,i-1}, w_i \rangle$,
where $B_{40,7}=C_{40,7}^{(1)}$ and $\supp(w_i)$ 
is listed in Table~\ref{Tab:40}.
The weight distributions of $B_{40,i}$ $(i=8,9,\ldots,19)$
are also listed in the table, where
$A_j$ denotes the number of codewords of weight $j$
($j=4,8,12,16,20$).
From the weight distributions,
one can easily verify that $w_i \not\in B_{40,i-1}$ and 
$B_{40,i}$ is doubly even for $i=8,9,\ldots,19$.
There are extremal Type~II $\ZZ_4$-codes with residue codes of dimension $20$.
By Lemma~\ref{Lem:wt4}, 
we have the following:

\begin{prop}\label{Prop:40}
There is an extremal Type~II $\ZZ_4$-code of length 
$40$ whose residue code has dimension $k$
if and only if $k \in \{7,8,\ldots,20\}$.
\end{prop}

\begin{table}[thb]
\centering
\caption{Supports $\supp(w_i)$ and weight distributions of $B_{40,i}$}
\vspace*{0.2in}
\label{Tab:40}
{\small
\begin{tabular}{c|l|rrrrr}
\noalign{\hrule height1pt}
$i$ & \multicolumn{1}{c|}{$\supp(w_i)$} 
&$A_4$ &$A_8$ & $A_{12}$& $A_{16}$ & $A_{20}$  \\
\hline
 8&$\{1,2, 4,29\}$&   1&    0&     1&     35&    180 \\
 9&$\{1,2, 5,33\}$&   3&    0&     3&     75&    348 \\
10&$\{1,2, 7,31\}$&   6&    1&    10&    150&    688 \\
11&$\{1,2, 9,10\}$&  10&    6&    22&    313&   1344 \\
12&$\{1,2,11,17\}$&  15&   21&    48&    634&   2658 \\
13&$\{1,2,12,39\}$&  22&   56&   102&   1271&   5288 \\
14&$\{1,2,13,27\}$&  29&   99&   280&   2620&  10326 \\
15&$\{1,2,14,37\}$&  37&  175&   688&   5296&  20374 \\
16&$\{1,2,15,35\}$&  47&  313&  1548&  10694&  40330 \\
17&$\{1,2,20,36\}$&  57&  509&  3436&  21698&  79670 \\
18&$\{1,2,21,28\}$&  68&  845&  7344&  43826& 157976 \\
19&$\{1,2,24,32\}$&  84& 1533& 15184&  87938& 314808 \\
\noalign{\hrule height1pt}
\end{tabular}
}
\end{table}

As another approach to Proposition~\ref{Prop:40},
we explicitly found an extremal Type~II $\ZZ_4$-code
$C_{40,i}$ with $C_{40,i}^{(1)} \cong B_{40,i}$
for $i=8,9,\ldots,19$.
To save space, we only list
in Figure~\ref{Fig:40} the $i \times (40-i)$ matrices
$A$ in generator matrices of the form (\ref{Eq:SF}).

\begin{rem}
Similar to Remark \ref{rem:32},
any of the codes $C_{40,i}$ $(i=7,8,\ldots,19)$
has minimum Hamming weight $4$ and
minimum Lee weight $8$.
\end{rem}

\subsection{Residue codes of dimension 7}

At lengths $24$ and $32$,
the smallest dimensions among binary codes
satisfying (\ref{eq:b1})--(\ref{eq:b3}) are both $6$,
and there is a unique extremal Type~II $\ZZ_4$-code
with residue code of dimension $6$,
up to equivalence, for both lengths
(see~\cite{HLM} and Proposition~\ref{Prop:32dim6}).

At length $40$,
we found an extremal Type~II $\ZZ_4$-code $C'_{40,7}$
with residue code $C'^{(1)}_{40,7}=
\langle C^{(1)}_{40,7} \cap \langle v \rangle^\perp, v \rangle$,
where
\[
\supp(v)=
\{ 1, 3, 4, 6, 8, 9, 10, 11, 12, 13, 18, 20 \}.
\]
The weight enumerators of 
$C'^{(1)}_{40,7}$ and ${C'^{(1)}_{40,7}}^\perp$ are
given by:
\begin{align*}
& 1+ y^{12}+  11 y^{16}
+ 102 y^{20}+  11 y^{24}+ y^{28}+ y^{40},
\\ &
1
+      1542 y^{4}
+     59264 y^{6}
+   1204653 y^{8}
+  13234816 y^{10}
+  87321928 y^{12}
\\ &
+ 362544000 y^{14}
+ 982186834 y^{16}
+1771383424 y^{18}
+2154061668 y^{20}
\\ &
+ \cdots + y^{40},
\end{align*}
respectively.
In order to give a generator matrix of $C'_{40,7}$
of the form (\ref{Eq:SF}), we only list
the $7 \times 33$ matrix $A$ in (\ref{Eq:SF}):
\[
A=
\left(\begin{array}{c}
 100000000000001011111111111030232\\
 011011011101000001011000101230302\\
 011100001110011110001110100311332\\
 100000111111113101101010010201033\\
 010110010110101100111101000312111\\
 010001111010000010000001011311013\\
 111111111111111000000000000020200
\end{array}\right).
\]
Hence, at length $40$, 
there are at least two inequivalent
extremal Type~II $\ZZ_4$-codes whose residue
codes have the smallest dimension
among binary codes satisfying (\ref{eq:b1})--(\ref{eq:b3}). 

Among these $23$ known codes,
the $22$ codes have residue codes which are 
doubly even self-dual codes
and the residue code of the other code given in~\cite{Z4-PSQ}
has dimension $13$ and
the following weight enumerator:
\[
1
+  156 y^{12}
+ 1911 y^{16}
+ 4056 y^{20}
+ 1911 y^{24}
+  156 y^{28}
+ y^{40}.
\]
It turns out that the code in~\cite{Z4-PSQ} and $C_{40,13}$
are inequivalent. 
Hence, none of 
the codes $C_{40,i}$ $(i=7,8,\ldots,19)$
and $C'_{40,7}$ is equivalent to any of the
known codes.
Thus, we have the following:

\begin{cor}
There are at least $37$ inequivalent
extremal Type~II $\ZZ_4$-codes of length $40$.
\end{cor}

The binary $[40,8]$ code $N_{40}=\langle C_{40,7}^{(1)},w \rangle$
satisfies (\ref{eq:b1})--(\ref{eq:b3}),
where 
\[
\supp(w)=\{4, 8, 13, 22, 23, 34, 36, 39\}.
\]
However, we verified that
none of the Type~II $\ZZ_4$-codes $C$ with 
$C^{(1)}=N_{40}$ is extremal,
using the method in Section~\ref{SS:method}.
Therefore, there is a binary code
satisfying (\ref{eq:b1})--(\ref{eq:b3}) 
which cannot be realized as the residue code of an 
extremal Type~II $\ZZ_4$-code of length $40$. 
It is not known whether there is a binary $[40,7]$ code $B$
satisfying (\ref{eq:b1})--(\ref{eq:b3})
such that none of the Type~II $\ZZ_4$-codes $C$ with 
$C^{(1)}=B$ is extremal.


\begin{figure}[p]
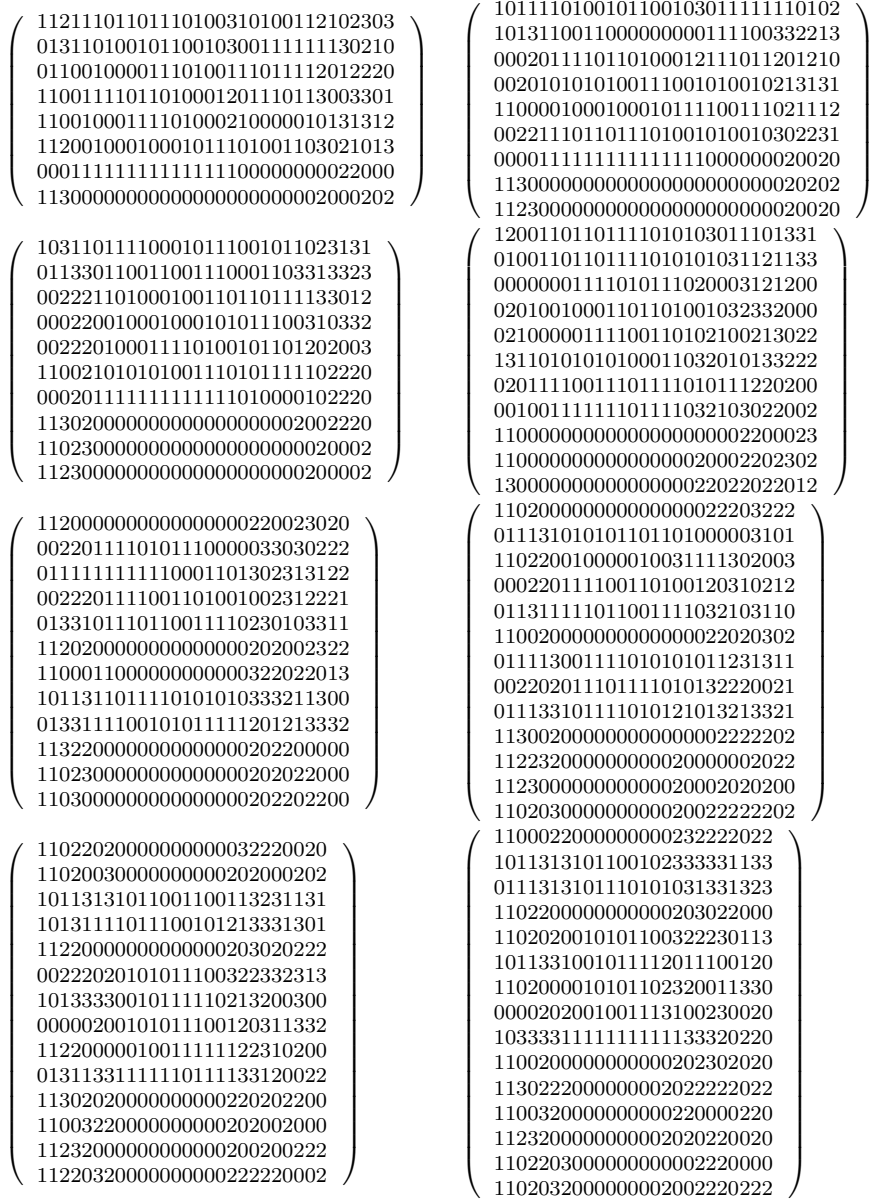

\centering
{\scriptsize
\begin{tabular}{ll}
$
\left(\begin{array}{c}
11211101101110100310100112102303\\
01311010010110010300111111130210\\
01100100001110100111011112012220\\
11001111011010001201110113003301\\
11001000111101000210000010131312\\
11200100010001011101001103021013\\
00011111111111111100000000022000\\
11300000000000000000000002000202
\end{array}\right)
$&$
\left(\begin{array}{c}
1011110100101100103011111110102\\
1013110011000000000111100332213\\
0002011110110100012111011201210\\
0020101010100111001010010213131\\
1100001000100010111100111021112\\
0022111011011101001010010302231\\
0000111111111111111000000020020\\
1130000000000000000000000020202\\
1123000000000000000000000020020
\end{array}\right)
$\\$
\left(\begin{array}{c}
103110111100010111001011023131\\
011330110011001110001103313323\\
002221101000100110110111133012\\
000220010001000101011100310332\\
002220100011110100101101202003\\
110021010101001110101111102220\\
000201111111111111010000102220\\
113020000000000000000002002220\\
110230000000000000000000020002\\
112300000000000000000000200002
\end{array}\right)
$&$
\left(\begin{array}{c}
12001101101111010103011101331\\
01001101101111010101031121133\\
00000001111010111020003121200\\
02010010001101101001032332000\\
02100000111100110102100213022\\
13110101010100011032010133222\\
02011110011101111010111220200\\
00100111111101111032103022002\\
11000000000000000000002200023\\
11000000000000000020002202302\\
13000000000000000022022022012
\end{array}\right)
$\\$
\left(\begin{array}{c}
1120000000000000000220023020\\
0022011110101110000033030222\\
0111111111110001101302313122\\
0022201111001101001002312221\\
0133101110110011110230103311\\
1120200000000000000202002322\\
1100011000000000000322022013\\
1011311011110101010333211300\\
0133111100101011111201213332\\
1132200000000000000202200000\\
1102300000000000000202022000\\
1103000000000000000202202200
\end{array}\right)
$&$
\left(\begin{array}{c}
110200000000000000022203222\\
011131010101101101000003101\\
110220010000010031111302003\\
000220111100110100120310212\\
011311111011001111032103110\\
110020000000000000022020302\\
011113001111010101011231311\\
002202011101111010132220021\\
011133101111010121013213321\\
113002000000000000002222202\\
112232000000000020000002022\\
112300000000000020002020200\\
110203000000000020022222202
\end{array}\right)
$\\$
\left(\begin{array}{c}
11022020000000000032220020\\
11020030000000000202000202\\
10113131011001100113231131\\
10131111011100101213331301\\
11220000000000000203020222\\
00222020101011100322332313\\
10133330010111110213200300\\
00000200101011100120311332\\
11220000010011111122310200\\
01311331111110111133120022\\
11302020000000000220202200\\
11003220000000000202002000\\
11232000000000000200200222\\
11220320000000000222220002
\end{array}\right)
$&$
\left(\begin{array}{c}
1100022000000000232222022\\
1011313101100102333331133\\
0111313101110101031331323\\
1102200000000000203022000\\
1102020010101100322230113\\
1011331001011112011100120\\
1102000010101102320011330\\
0000202001001113100230020\\
1033331111111111133320220\\
1100200000000000202302020\\
1130222000000002022222022\\
1100320000000000220000220\\
1123200000000002020220020\\
1102203000000000002220000\\
1102032000000002002220222
\end{array}\right)
$
\end{tabular}
\caption{
Matrices $A$ in generator matrices of $C_{40,i}$}
\label{Fig:40}
}
\end{figure}

\begin{figure}[thb]
\centering
{\scriptsize
\begin{tabular}{ll}
$
\left(\begin{array}{c}
112020000002002220302220\\
110020200003002022020020\\
000200000000000333020002\\
112002000000002222232002\\
011133301113001231111331\\
101113111103102220313102\\
101331111101100220313212\\
000022200100111132200203\\
112002001010112231002313\\
103333111113113310331200\\
112222000002002220003002\\
113002000002000202000002\\
112232000000000022020020\\
112302000002000222220200\\
112220300000000000222002\\
110003200002000000200022
\end{array}\right)
$
&
$
\left(\begin{array}{c}
11202002200000220232020\\
00000222200020033302020\\
11220000000020220003202\\
10313133101110321311131\\
10111331311100200013332\\
10313311311120220031321\\
11220020200121111222220\\
00222220001011023300011\\
10311311311131113031302\\
11200022000020022020322\\
11322002000000202200020\\
11023020200000200002220\\
11030200200000220220200\\
11022003200020200202000\\
11200220300000220002200\\
11220032000020202022002\\
11202322000020020022200
\end{array}\right)
$
\\
$
\left(\begin{array}{c}
1100200220022202222302\\
0010210300202202002222\\
1300220020022220001202\\
0111231033331313331313\\
1310301330020222022220\\
1310100312010000000002\\
0211001113232220002022\\
1311011200232020002200\\
1300000000220222212202\\
1211031013133113131113\\
1300020222022012200020\\
1300200202220201202200\\
1301130320322020202000\\
1300200222222222000201\\
1100220022203020002022\\
1300000222022002022012\\
1100020020222302022002\\
1100000002202222320000
\end{array}\right)
$
&
$
\left(\begin{array}{c}
013002020002021202002\\
002330022000030000022\\
011200202202002320000\\
113030002100030202302\\
113212220120232221222\\
003133313011121330213\\
113230022222012003302\\
013222220200222212022\\
011222222202000222030\\
013002202012222022000\\
101233113131113113113\\
013200200202102002002\\
013002002222000002201\\
013220022201020220020\\
013202012020220002022\\
011022023222020002002\\
013220120202022202022\\
102312200022002022222\\
011003020220022222002
\end{array}\right)
$
\end{tabular}
\setcounter{figure}{1}
\caption{
Matrices $A$ in generator matrices of $C_{40,i}$ (continued)}
}
\end{figure}

\bigskip
\noindent
{\bf Acknowledgment.} 
The author would like to thank Akihiro Munemasa
for his help in 
the classification given in Proposition~\ref{Prop:32dim6}.


\end{document}